\documentclass[11pt]{article}

\usepackage[T1]{fontenc}
\usepackage{lmodern}
\usepackage{amsmath,amssymb,amsthm,mathtools,bm}
\usepackage{booktabs,array}
\usepackage{enumitem}
\usepackage{geometry}
\usepackage{xcolor}
\usepackage{microtype}
\usepackage{graphicx}
\usepackage{algorithm}
\usepackage{algpseudocode}
\usepackage{hyperref}

\hypersetup{colorlinks=true,linkcolor=blue,citecolor=blue,urlcolor=blue}
\allowdisplaybreaks

\newtheorem{theorem}{Theorem}[section]
\newtheorem{proposition}[theorem]{Proposition}
\newtheorem{lemma}[theorem]{Lemma}
\newtheorem{corollary}[theorem]{Corollary}
\theoremstyle{definition}
\newtheorem{definition}[theorem]{Definition}
\newtheorem{assumption}[theorem]{Assumption}

\DeclareMathOperator{\rank}{rank}
\DeclareMathOperator{\diag}{diag}
\DeclareMathOperator*{\argmin}{arg\,min}
\DeclareMathOperator{\Log}{Log}

\newcommand{\cN}{\mathcal N}
\newcommand{\C}{\mathbb C}
\newcommand{\R}{\mathbb R}
\newcommand{\Id}{\mathrm I}
\newcommand{\eps}{\varepsilon}
\newcommand{\Xop}{\mathsf X}

\title{\bfseries Dilation-Covariant Hankel Pencils for Multiscale Recovery of Sparse Mellin Spectra}
\author{
Zhiliang Deng\thanks{Corresponding author. School of Mathematical Science, University of Electronic Science and Technology of China, dengzhl@uestc.edu.cn}
\and Xiaomei Yang\thanks{School of Mathematics, Southwest Jiaotong University, yangxiaomath@swjtu.edu.cn}}
\date{}

\begin{document}
\maketitle

\begin{abstract}
We study sparse Mellin spectral recovery from geometric samples.  
Dilation covariance is shown to generate the Hankel structure directly, rather than only after reduction to a classical exponential-sum model.  For complex exponents, incommensurate sampling scales remove the logarithmic aliasing that persists at a single scale.  A second recovery channel is obtained from the minimal Hankel pencil: contour winding counts spectral nodes regionally, while Rouch\'{e}-type bounds certify isolated nodes and unresolved clusters under noise.  The associated contour margin has an explicit degeneration rate as the sampling ratio approaches one.  Numerical experiments show that the auxiliary scale improves both identifiability and resolution, and that certified contour counts can remain reliable beyond the regime of accurate pointwise recovery.
\end{abstract}



\section{Introduction}

Scale-dependent data on the positive half-line are naturally tied to the multiplicative structure of $\mathbb R_+$.  The logarithmic change of variables $t=\log x$ identifies this geometry with the additive geometry of $\mathbb R$: dilations become translations, and the Mellin transform becomes the Fourier transform \cite{ButzerJansche1997,Folland1995}.  For a sparse Mellin model
\begin{equation}\label{eq:mellin-sparse-model}
f(x)=\sum_{\ell=1}^{r}a_\ell x^{\alpha_\ell},
\qquad
a_\ell,\alpha_\ell\in\mathbb C,
\end{equation}
sampling on the geometric grid $x_n=x_0q^n$, with $x_0>0$ and $q\in(0,1)$, gives
\begin{equation}\label{eq:mellin-geometric-samples}
y_n^{(q)}
=f(x_0q^n)
=\sum_{\ell=1}^{r}a_\ell x_0^{\alpha_\ell}
\bigl(q^{\alpha_\ell}\bigr)^n.
\end{equation}
Thus, at a fixed scale, the recovery problem is a finite exponential-sum problem with nodes $q^{\alpha_\ell}$.  Its finite-rank Hankel matrices, annihilating polynomial, and matrix-pencil reconstruction belong to the classical Prony framework \cite{Prony1795}.  Operator-based versions of Prony's method place the same mechanism in a broader spectral-identification setting \cite{StampferPlonka2020,KellerPlonka2021}; the Banach-module formulation of Krishtal and Pfander \cite{KrishtalPfander2025} is a recent development in this direction.

The exponential-sum reduction is not the only route to the Hankel structure.  In the multiplicative setting, a homogeneous operator and a discrete dilation satisfy a commutation relation whose ordered products carry a scalar multiplier.  Removing this multiplier restores the additive index $i+j$ and produces a Hankel moment matrix before the finite Mellin model is imposed.  In the realization on $\mathbb R_+$, these normalized moments are the geometric samples in \eqref{eq:mellin-geometric-samples}; the usual Prony factorization then follows by specializing to the sparse spectrum \eqref{eq:mellin-sparse-model}.

For complex exponents, the sampling ratio also affects identifiability.  At one scale, $\alpha\mapsto q^\alpha$ is periodic in the imaginary direction, so a recovered node determines $\alpha$ only modulo a logarithmic lattice.  Several scales reduce the ambiguity to the intersection of their aliasing lattices, and two logarithmic steps with irrational ratio remove it globally.  A finite spectrum introduces the additional problem of pairing the unordered node sets recovered on the separate grids.  Distinct effective amplitudes provide the pairing directly, while repeated amplitudes can be treated through a compatibility gap in the joint scale embedding.  The resulting two-scale uniqueness statement differs from multiscale matrix-pencil constructions used primarily for conditioning or computational efficiency \cite{CuytLee2024}: here an auxiliary scale can change exact identifiability of a complex Mellin spectrum.

The stability analysis has two parts.  Singular-value estimates describe perturbation of the finite-rank Hankel structure and the loss of separation as $q\to1$; the same Vandermonde effects underlie super-resolution and clustered exponential recovery \cite{Liao2016,Batenkov2018,LiLiao2021,KatzDiabBatenkov2024,KunisNagel2020}.  A complementary description is obtained from the determinant of the minimal Hankel pencil.  Its zeros are the Mellin nodes, so the argument principle gives regional node counts, while Rouch\'{e}'s theorem supplies perturbation conditions under which those counts persist.  Small contours localize isolated nodes; larger contours certify the cardinality of unresolved clusters.  This leads to a dual-channel procedure: a projected-pencil channel estimates individual nodes, whereas a contour channel computes winding numbers directly and, when a deterministic noise bound is available, attaches an a posteriori certificate to the regional count.

The same characteristic quantifies the degeneration of contour stability near $q=1$.  The Mellin nodes coalesce at one, the square Hankel determinant vanishes at an explicit order, and the fixed-contour margin inherits the corresponding power of $|\log q|$.  The numerical experiments examine finite-rank recovery, the $q\to1$ degeneration, multiscale de-aliasing, clustered spectra under noise, and the distinct failure modes of the point and contour channels.

Section~\ref{sec:qweyl} derives the covariance-induced Hankel lifting and its geometric realization.  Section~\ref{sec:identifiability} treats multiscale aliasing, matching, and exact recovery.  Section~\ref{sec:stability} develops the perturbation and contour-stability analysis, and Section~\ref{sec:algorithm} gives the reconstruction procedure and numerical experiments.  Appendix~\ref{sec:qfock} records a $q$-Fock realization of the same commutation structure.

\section{Dilation covariance and Hankel lifting}
\label{sec:qweyl}

The multiplicative geometry is represented naturally on $L^2(\mathbb R_+,dx/x)$.  For $a>0$, let
$$
(D_af)(x)=f(ax).
$$
The operators $D_a$ are unitary, and the logarithmic map
$$
(\mathcal Uf)(t)=f(e^t)
$$
conjugates them to translations:
\begin{equation}\label{eq:dilation-translation}
\mathcal U D_a\mathcal U^{-1}=\tau_{\log a},
\qquad
(\tau_hg)(t)=g(t+h).
\end{equation}
Under the same identification, the Mellin transform is the Fourier transform in the additive variable and diagonalizes the dilation representation \cite{ButzerJansche1997,Folland1995}.  The generalized power modes satisfy
$$
D_ax^\alpha=a^\alpha x^\alpha,
\qquad
x^\alpha=e^{\alpha\log x},
$$
as an algebraic identity; a general power function need not belong to $L^2(\mathbb R_+,dx/x)$.

We now isolate the covariance relation from this realization.  Let $\mathcal X$ be a complex linear space and let $\{U_t\}_{t\in\mathbb R}$ be a representation of $(\mathbb R,+)$ by invertible operators.  Let $\mathcal D\subset\mathcal X$ be invariant under every $U_t$, and suppose that $M:\mathcal D\to\mathcal D$ satisfies
\begin{equation}\label{eq:homogeneous-covariance}
U_tMU_{-t}=e^{\kappa t}M
\end{equation}
for some $\kappa\in\mathbb C$.  Fix $h\ne0$, set $q=e^h>0$, and write $T_q=U_h$.  With the convention $q^\kappa=e^{\kappa h}$, \eqref{eq:homogeneous-covariance} becomes
\begin{equation}\label{eq:q-weyl-pair}
T_qM=q^\kappa MT_q.
\end{equation}
For $\kappa=1$ this is the usual $q$-Weyl commutation relation.

\begin{lemma}\label{lem:ordered}
For $j,k\in\mathbb N_0$,
\begin{equation}\label{eq:ordered}
T_q^jM^k=q^{\kappa jk}M^kT_q^j.
\end{equation}
\end{lemma}

\begin{proof}
Repeated application of \eqref{eq:q-weyl-pair} gives the identity.
\end{proof}

Set
\begin{equation}\label{eq:En}
E_n=M^nT_q^n,\qquad n\ge0.
\end{equation}

\begin{proposition}[Projective multiplication law]\label{prop:cocycle}
For all $i,j\ge0$,
\begin{equation}\label{eq:EiEj}
E_iE_j=q^{\kappa ij}E_{i+j}.
\end{equation}
The multiplier
$$
\omega_{q,\kappa}(i,j)=q^{\kappa ij}
$$
is normalized and satisfies
\begin{equation}\label{eq:cocycle-id}
\omega_{q,\kappa}(i,j)\omega_{q,\kappa}(i+j,k)
=\omega_{q,\kappa}(j,k)\omega_{q,\kappa}(i,j+k).
\end{equation}
Thus the additive monoid law on $\mathbb N_0$ is represented by $\{E_n\}$ up to the scalar multiplier $\omega_{q,\kappa}$.
\end{proposition}

\begin{proof}
Lemma~\ref{lem:ordered} gives
$$
E_iE_j=M^iT_q^iM^jT_q^j
=q^{\kappa ij}M^{i+j}T_q^{i+j}.
$$
Both sides of \eqref{eq:cocycle-id} equal $q^{\kappa(ij+ik+jk)}$.
\end{proof}

Let $\mathcal L:\mathcal D\to\mathbb C$ be linear and let $f\in\mathcal D$.  The quantities $\mathcal L(E_iE_jf)$ contain the factor $q^{\kappa ij}$ and therefore need not depend on $i+j$ alone.

\begin{definition}[Cocycle-normalized Hankel lifting]\label{def:q-hankel}
For $m\ge1$, set
\begin{equation}\label{eq:q-hankel-def}
\mathcal H_m^{(q,\kappa)}[f;\mathcal L]
:=\left(q^{-\kappa ij}\mathcal L(E_iE_jf)\right)_{i,j=0}^{m-1}.
\end{equation}
By Proposition~\ref{prop:cocycle},
\begin{equation}\label{eq:q-hankel-standard}
\mathcal H_m^{(q,\kappa)}[f;\mathcal L]
=\left(\mathcal L(E_{i+j}f)\right)_{i,j=0}^{m-1}.
\end{equation}
\end{definition}

\begin{proposition}\label{prop:normalization-unique}
Suppose scalars $c_{ij}$ satisfy
$$
c_{ij}E_iE_j=E_{i+j}
$$
whenever $E_{i+j}\ne0$.  Then $c_{ij}=q^{-\kappa ij}$.
\end{proposition}

\begin{proof}
By \eqref{eq:EiEj},
$c_{ij}q^{\kappa ij}E_{i+j}=E_{i+j}$, and the claim follows whenever $E_{i+j}\ne0$.
\end{proof}

Hence $q^{-\kappa ij}$ is the unique entrywise scalar normalization that restores the additive index $i+j$; it is not an independent reweighting of the moments.

We next specialize to geometric sampling.  Fix $x_0>0$ and $q\in(0,1)$, and let $\mathcal F$ be a linear space of pointwise-defined functions on $\mathbb R_+$ that is invariant under
\begin{equation}\label{eq:scale-ops}
(T_qf)(x)=f(qx),\qquad
(\Xop_{x_0}f)(x)=\frac{x}{x_0}f(x).
\end{equation}
We also require evaluation at $x_0$ to be defined on $\mathcal F$; finite linear combinations of power functions suffice for the sparse model.  Since
$$
T_q\Xop_{x_0}=q\Xop_{x_0}T_q,
$$
this is the case $\kappa=1$ of the abstract construction.  With
$$
\mathcal L_{x_0}(g)=g(x_0),
\qquad
E_n=\Xop_{x_0}^nT_q^n,
$$
we obtain the following realization.

\begin{proposition}\label{prop:geometric-representation}
For $f\in\mathcal F$,
\begin{equation}\label{eq:En-action}
(E_nf)(x)=\left(\frac{x}{x_0}\right)^nf(q^nx),
\end{equation}
and hence
\begin{equation}\label{eq:En-eval}
\mathcal L_{x_0}(E_nf)=f(q^nx_0).
\end{equation}
Moreover,
\begin{equation}\label{eq:q-hankel-geometric}
q^{-ij}\mathcal L_{x_0}(E_iE_jf)=f(q^{i+j}x_0).
\end{equation}
\end{proposition}

\begin{proof}
The first identity follows from $T_q^nf(x)=f(q^nx)$ and
$\Xop_{x_0}^ng(x)=(x/x_0)^ng(x)$.  Evaluation at $x_0$ gives \eqref{eq:En-eval}, and \eqref{eq:q-hankel-geometric} follows from Proposition~\ref{prop:cocycle}.
\end{proof}

Thus the abstract lifting becomes the geometric Hankel matrix
\begin{equation}\label{eq:scale-hankel}
\mathcal H_m^{(q)}[f]
:=\left(f(q^{i+j}x_0)\right)_{i,j=0}^{m-1}.
\end{equation}
The Hilbert-space realization and the sampling functional play different roles here.  The dilation representation is naturally defined on $L^2(\mathbb R_+,dx/x)$, whereas point evaluation is not bounded on that space; the latter is used only on the pointwise class $\mathcal F$.

For the sparse model \eqref{eq:mellin-sparse-model}, write
\begin{equation}\label{eq:w-rho}
w_\ell=a_\ell x_0^{\alpha_\ell},
\qquad
\rho_\ell(q)=q^{\alpha_\ell}=e^{\alpha_\ell\log q}.
\end{equation}
Then
\begin{equation}\label{eq:moments-scale}
\mu_n^{(q)}:=f(x_0q^n)
=\sum_{\ell=1}^{r}w_\ell\rho_\ell(q)^n.
\end{equation}
When several grids have the same base point $x_0$, the weights $w_\ell$ are common to all scales, whereas the nodes $\rho_\ell(q)$ depend on $q$.

From this point on we use $H_{m,s}^{(q)}$ for finite Hankel blocks, with the first subscript denoting the matrix size and the second the shift:
\begin{equation}\label{eq:shifted-hankel}
H_{m,s}^{(q)}
:=\left(\mu_{i+j+s}^{(q)}\right)_{i,j=0}^{m-1},
\qquad m\ge1,\quad s\ge0.
\end{equation}
In particular, $H_{m,0}^{(q)}=\mathcal H_m^{(q)}[f]$.  Let
$$
V_m(\rho(q))
=
\bigl(\rho_\ell(q)^k\bigr)_{
0\le k\le m-1,\;1\le\ell\le r},
\qquad
W=\diag(w_1,\ldots,w_r),
\qquad
\Theta_q=\diag\bigl(\rho_1(q),\ldots,\rho_r(q)\bigr).
$$
The superscript $T$ denotes the ordinary transpose rather than the conjugate transpose.

\begin{proposition}\label{prop:scale-factorization}
Assume that $\rho_1(q),\ldots,\rho_r(q)$ are pairwise distinct and $w_\ell\ne0$.  Then, for every $s\ge0$,
\begin{equation}\label{eq:scale-factorization}
H_{m,s}^{(q)}
=V_m(\rho(q))W\Theta_q^sV_m(\rho(q))^T.
\end{equation}
If $m\ge r$, then $\rank H_{m,s}^{(q)}=r$.  For $m=r$ and $s=0$,
\begin{equation}\label{eq:det-scale}
\det H_{r,0}^{(q)}
=\left(\prod_{\ell=1}^{r}w_\ell\right)
\prod_{1\le k<\ell\le r}
\bigl(\rho_\ell(q)-\rho_k(q)\bigr)^2.
\end{equation}
\end{proposition}

\begin{proof}
Equation \eqref{eq:moments-scale} gives \eqref{eq:scale-factorization} entrywise.  The rank and determinant statements follow from the Vandermonde matrix and the nonvanishing of the diagonal factors.
\end{proof}

At a fixed scale, define the Prony annihilating polynomial
\begin{equation}\label{eq:prony-polynomial}
P_q(z)
:=\prod_{\ell=1}^{r}\bigl(z-\rho_\ell(q)\bigr)
=\sum_{j=0}^{r}p_j^{(q)}z^j.
\end{equation}
Since $P_q(\rho_\ell(q))=0$,
\begin{equation}\label{eq:annihilation}
\sum_{j=0}^{r}p_j^{(q)}\mu_{n+j}^{(q)}=0,
\qquad n\ge0.
\end{equation}
Under the hypotheses of Proposition~\ref{prop:scale-factorization}, $P_q$ is the minimal annihilating polynomial of the moment sequence.

The matrix-pencil formulation uses the two consecutive $r\times r$ blocks already contained in \eqref{eq:shifted-hankel}.  Set
\begin{equation}\label{eq:pencil-characteristic}
A_q(z):=H_{r,1}^{(q)}-zH_{r,0}^{(q)},
\qquad
\Psi_q(z):=\det A_q(z).
\end{equation}
Their factorizations give
\begin{equation}\label{eq:pencil-characteristic-factor}
A_q(z)
=V_rW(\Theta_q-z\Id)V_r^T,
\qquad
\Psi_q(z)
=\det H_{r,0}^{(q)}\prod_{\ell=1}^{r}\bigl(\rho_\ell(q)-z\bigr).
\end{equation}
Equivalently,
\begin{equation}\label{eq:psi-prony-relation}
\Psi_q(z)=(-1)^r\det H_{r,0}^{(q)}P_q(z).
\end{equation}
Thus the generalized eigenvalues of the pencil $A_q(z)$, the zeros of $\Psi_q$, and the zeros of the Prony polynomial $P_q$ are the same nodes $\rho_\ell(q)$.

For real exponents,
\begin{equation}\label{eq:alpha-from-rho}
\alpha_\ell=\frac{\log\rho_\ell(q)}{\log q},
\qquad
a_\ell=w_\ell x_0^{-\alpha_\ell}.
\end{equation}
For complex exponents the logarithm is multivalued, which leads to the multiscale question considered next.

\section{Multiscale identifiability and exact recovery}
\label{sec:identifiability}

At one scale, recovering the node $\rho=q^\alpha$ does not determine a complex exponent uniquely.  Writing $h=\log q<0$, define
\begin{equation}\label{eq:single-embedding}
\Phi_h(\alpha)=e^{h\alpha}=q^\alpha.
\end{equation}
The periodicity of the exponential gives the complete single-scale ambiguity.

\begin{proposition}\label{prop:single-alias}
For $\alpha,\beta\in\C$,
\begin{equation}\label{eq:single-alias}
q^\alpha=q^\beta
\quad\Longleftrightarrow\quad
\alpha-\beta\in\frac{2\pi i}{\log q}\mathbb Z.
\end{equation}
\end{proposition}

\begin{proof}
The equality $q^\alpha=q^\beta$ is equivalent to $e^{h(\alpha-\beta)}=1$, hence to $h(\alpha-\beta)=2\pi ik$ for some $k\in\mathbb Z$.
\end{proof}

One may select a representative by restricting to the fundamental strip
\begin{equation}\label{eq:fundamental-strip}
\Omega_q
=\left\{\alpha\in\C:
-\frac{\pi}{|\log q|}<\operatorname{Im}\alpha
\le\frac{\pi}{|\log q|}\right\},
\end{equation}
but this is an a priori branch restriction rather than a resolution supplied by the data.

For scales $q_j\in(0,1)$, set $h_j=\log q_j$ and define
\begin{equation}\label{eq:joint-embedding}
\Phi_{\bm h}(\alpha)
=\bigl(e^{h_1\alpha},\ldots,e^{h_s\alpha}\bigr),
\end{equation}
together with the joint aliasing lattice
\begin{equation}\label{eq:joint-lattice}
\Lambda(\bm h)
:=\left\{\delta\in\C:
h_j\delta\in2\pi i\mathbb Z
\text{ for }j=1,\ldots,s\right\}.
\end{equation}

\begin{theorem}[Multiscale aliasing classification]\label{thm:multi-alias}
For $\alpha,\beta\in\C$,
$$
\Phi_{\bm h}(\alpha)=\Phi_{\bm h}(\beta)
\quad\Longleftrightarrow\quad
\alpha-\beta\in\Lambda(\bm h).
$$
Moreover:
\begin{enumerate}[label=(\roman*)]
\item if $h_j/h_k\notin\mathbb Q$ for at least one pair $j\ne k$, then $\Lambda(\bm h)=\{0\}$ and $\Phi_{\bm h}$ is injective on $\C$;
\item if $h_j=n_jd$ with $d\ne0$ and nonzero relatively prime integers $n_1,\ldots,n_s$, then
\begin{equation}\label{eq:commensurate-lattice}
\Lambda(\bm h)=\frac{2\pi i}{d}\mathbb Z.
\end{equation}
\end{enumerate}
\end{theorem}

\begin{proof}
The first assertion follows from $e^{h_j(\alpha-\beta)}=1$ for every $j$.  If $0\ne\delta\in\Lambda(\bm h)$, then $h_j\delta=2\pi ik_j$ with $k_j\ne0$, and hence
$$
\frac{h_j}{h_k}=\frac{k_j}{k_k}\in\mathbb Q.
$$
This proves (i).  For (ii), $2\pi im/d$ belongs to $\Lambda(\bm h)$ for every $m\in\mathbb Z$.  Conversely, if $n_jd\delta/(2\pi i)\in\mathbb Z$ for every $j$, B\'{e}zout's identity and $\gcd(n_1,\ldots,n_s)=1$ imply $d\delta/(2\pi i)\in\mathbb Z$.
\end{proof}

In the commensurate case the ambiguity is reduced but not removed: the joint map is injective only after restricting the imaginary part to a fundamental strip of width $2\pi/|d|$.  Thus rationally related scales can enlarge the unambiguous range without giving global uniqueness on $\mathbb C$.

The theorem concerns a single exponent.  A finite spectrum introduces a second issue: the nodes recovered independently at different scales are unordered.  With a common base point $x_0$, the data at scale $q_j$ are
\begin{equation}\label{eq:multi-samples}
y_n^{(j)}
=\sum_{\ell=1}^{r}w_\ell\rho_{\ell,j}^n,
\qquad
\rho_{\ell,j}=q_j^{\alpha_\ell},
\end{equation}
where $w_\ell=a_\ell x_0^{\alpha_\ell}$ is independent of $j$.  Exact Prony recovery on each grid therefore returns the unordered set
\begin{equation}\label{eq:unordered-pairs}
\mathcal P_j
=\left\{(\rho_{\ell,j},w_\ell):\ell=1,\ldots,r\right\}.
\end{equation}

If the effective weights $w_1,\ldots,w_r$ are pairwise distinct, their scale independence determines the cross-scale pairing uniquely: each set $\mathcal P_j$ contains exactly one pair with second component $w_\ell$.  When weights repeat, the node geometry itself can supply the pairing.  For two scales and an admissible set $\Omega\subset\C$, define
\begin{equation}\label{eq:compat-cost}
d_\Omega(z_1,z_2)
:=\inf_{\alpha\in\Omega}
\left(|e^{h_1\alpha}-z_1|^2+|e^{h_2\alpha}-z_2|^2\right)^{1/2}.
\end{equation}
A correctly paired noiseless node pair has zero cost whenever the corresponding exponent lies in $\Omega$.

\begin{assumption}\label{ass:matching-gap}
There exists $\gamma>0$ such that every true pair has zero compatibility cost, while every false pair satisfies
$$
d_\Omega(\rho_{k,1},\rho_{\ell,2})\ge\gamma,
\qquad k\ne\ell.
$$
\end{assumption}

\begin{proposition}\label{prop:compat-matching}
Under Assumption~\ref{ass:matching-gap}, the true pairing is the unique minimum-total-cost bipartite matching.  Suppose, more generally, that nonnegative perturbed costs $\widetilde d_{k\ell}$ satisfy
$$
|\widetilde d_{k\ell}-d_\Omega(\rho_{k,1},\rho_{\ell,2})|\le\delta
$$
for all $k,\ell$.  If
\begin{equation}\label{eq:matching-perturbation}
\delta<\frac{\gamma}{r+1},
\end{equation}
then the true pairing remains the unique minimum-total-cost matching.
\end{proposition}

\begin{proof}
The exact true matching has total cost zero, while every other matching contains a false edge and hence has cost at least $\gamma$.  Under perturbation, the true matching has cost at most $r\delta$.  Any false matching contains an edge of cost at least $\gamma-\delta$, and all remaining costs are nonnegative.  Condition \eqref{eq:matching-perturbation} gives $r\delta<\gamma-\delta$.
\end{proof}

\begin{theorem}[Exact multiscale recovery]\label{thm:exact-multiscale}
Let
$$
f(x)=\sum_{\ell=1}^{r}a_\ell x^{\alpha_\ell}
$$
have nonzero amplitudes and pairwise distinct exponents.  Choose two scales $q_1,q_2\in(0,1)$ and collect $2r$ consecutive samples on each geometric grid with the same base point $x_0$.  Assume that
\begin{enumerate}[label=(\roman*)]
\item the nodes $q_j^{\alpha_\ell}$ are pairwise distinct for each $j=1,2$;
\item $\log q_1/\log q_2\notin\mathbb Q$;
\item either the effective weights $w_\ell=a_\ell x_0^{\alpha_\ell}$ are pairwise distinct, or there is an admissible set $\Omega$ containing the exponents for which Assumption~\ref{ass:matching-gap} holds.
\end{enumerate}
Then the representation is uniquely determined by the two data sets, up to permutation of its components.  Its minimal model order is $r$, and all $\alpha_\ell$ and $a_\ell$ are uniquely determined.
\end{theorem}

\begin{proof}
By (i) and the nonvanishing of the weights, each single-scale sequence has minimal Prony order $r$; its $2r$ consecutive samples therefore determine the $r$ nodes and effective weights.  Condition (iii) fixes the cross-scale pairing.  Theorem~\ref{thm:multi-alias} and (ii) then give a unique exponent for each paired node vector $(q_1^{\alpha_\ell},q_2^{\alpha_\ell})$.  Finally, $a_\ell=w_\ell x_0^{-\alpha_\ell}$.
\end{proof}

For real exponents, $\alpha\mapsto q^\alpha$ is already injective at one scale, so additional scales affect resolution rather than uniqueness.  For complex exponents they can affect both.

\section{Stability, contour counting, and scale design}
\label{sec:stability}

\subsection{Hankel perturbation and scale separation}

We first separate two sources of instability: perturbation of the finite-rank Hankel structure and compression of distinct exponents under $\alpha\mapsto q^\alpha$.  Suppose
\begin{equation}\label{eq:noisy-samples}
\widetilde y_n^{(q)}=y_n^{(q)}+\eta_n^{(q)},
\qquad n=0,\ldots,2m-2,
\end{equation}
and let
\begin{equation}\label{eq:noisy-hankel-zero-shift}
\widetilde H_{m,0}^{(q)}
=H_{m,0}^{(q)}+\Delta_{m,0}^{(q)},
\qquad
(\Delta_{m,0}^{(q)})_{ij}=\eta_{i+j}^{(q)}.
\end{equation}

If
$$
\max_{0\le n\le2m-2}|\eta_n^{(q)}|\le\eps,
$$
then each entry of $\Delta_{m,0}^{(q)}$ has modulus at most $\eps$, and hence
\begin{equation}\label{eq:hankel-noise-bound}
\|\Delta_{m,0}^{(q)}\|_2
\le\|\Delta_{m,0}^{(q)}\|_F
\le m\eps.
\end{equation}
If $H_{m,0}^{(q)}$ has rank $r$ and
$$
\|\Delta_{m,0}^{(q)}\|_2
<\frac12\sigma_r(H_{m,0}^{(q)}),
$$
Weyl's singular-value inequality yields
\begin{equation}\label{eq:sv-gap}
\sigma_r(\widetilde H_{m,0}^{(q)})
>\frac12\sigma_r(H_{m,0}^{(q)}),
\qquad
\sigma_{r+1}(\widetilde H_{m,0}^{(q)})
<\frac12\sigma_r(H_{m,0}^{(q)}).
\end{equation}
Thus a perturbation smaller than half the smallest nonzero singular value preserves a visible rank gap.

Under the assumptions of Proposition~\ref{prop:scale-factorization},
\begin{equation}\label{eq:sv-lower}
\sigma_r(H_{m,0}^{(q)})
\ge \sigma_r(V_m(\rho(q)))^2\min_\ell|w_\ell|.
\end{equation}
Thus rank resolution depends on the weakest effective weight and on the smallest singular value of the Vandermonde matrix.  Clustered nodes deteriorate the latter at a power-law rate in the relevant super-resolution regimes \cite{LiLiao2021,KunisNagel2020}.

For real exponents the scale dependence is already visible in the map $\alpha\mapsto q^\alpha$.  If $A\le\alpha,\beta\le B$ and $h=\log q<0$, the mean-value theorem gives
\begin{equation}\label{eq:separation-bounds}
|\log q|q^B|\alpha-\beta|
\le |q^\alpha-q^\beta|
\le |\log q|q^A|\alpha-\beta|.
\end{equation}
Conversely, if $\rho=q^\alpha>0$, $\widehat\rho>0$, and $|\widehat\rho-\rho|\le\rho/2$, then
\begin{equation}\label{eq:alpha-error}
|\widehat\alpha-\alpha|
\le\frac{2}{|\log q|\rho}|\widehat\rho-\rho|.
\end{equation}
For $\alpha\in[A,B]$ this yields
\begin{equation}\label{eq:alpha-error-uniform}
|\widehat\alpha-\alpha|
\le\frac{2}{|\log q|q^B}|\widehat\rho-\rho|.
\end{equation}
The common factor suggests the scale surrogate
\begin{equation}\label{eq:design-surrogate}
S_B(q)=|\log q|q^B.
\end{equation}

Writing $t=-\log q>0$ gives $S_B(q)=te^{-Bt}$, so for $B>0$ its unique maximizer is
\begin{equation}\label{eq:q-star}
q_*=e^{-1/B},
\end{equation}
with maximum $1/(eB)$.  This surrogate controls only the local scalar factor in \eqref{eq:separation-bounds}--\eqref{eq:alpha-error-uniform}; it does not optimize the full Vandermonde condition number.  Sample length, dynamic range, the exponent interval, and the noise model can shift the best practical scale.

For several real scales, applying the lower bound in \eqref{eq:separation-bounds} coordinatewise gives
\begin{equation}\label{eq:joint-separation}
\|\Phi_{\bm h}(\alpha)-\Phi_{\bm h}(\beta)\|_2
\ge |\alpha-\beta|
\left(\sum_{j=1}^{s}h_j^2e^{2h_jB}\right)^{1/2},
\qquad
\alpha,\beta\in[A,B]\subset\R,
\end{equation}
where $\bm h=(h_1,\ldots,h_s)$ and $h_j<0$.

The square Hankel determinant gives a complementary description of the singular regime $q\to1$.  With $h=\log q$,
$$
e^{h\alpha_\ell}-e^{h\alpha_k}
=h(\alpha_\ell-\alpha_k)+O(h^2).
$$

\begin{corollary}[Determinant degeneration as $q\to1$]\label{cor:q-one-det}
Assume that $\alpha_1,\ldots,\alpha_r\in\mathbb C$ are pairwise distinct and $w_\ell\ne0$.  Then
\begin{equation}\label{eq:q-one-det}
\det H_{r,0}^{(q)}
=\left(\prod_{\ell=1}^{r}w_\ell\right)
(\log q)^{r(r-1)}
\prod_{k<\ell}(\alpha_\ell-\alpha_k)^2
\bigl(1+O(|\log q|)\bigr)
\end{equation}
as $q\to1$ through positive real values.
\end{corollary}

\begin{proof}
Insert
$\rho_\ell(q)-\rho_k(q)
=(\log q)(\alpha_\ell-\alpha_k)+O((\log q)^2)$
into \eqref{eq:det-scale}.  There are $r(r-1)/2$ pairwise differences, each squared in the determinant.
\end{proof}

\subsection{Argument-principle counting and Rouch\'{e} stability}

We now return to the characteristic $\Psi_q$ in \eqref{eq:pencil-characteristic}.  The use of the minimal $r\times r$ pencil is essential: for an unreduced $m\times m$ exact Hankel pencil with $m>r$, both blocks have rank at most $r$, so its determinant vanishes identically.  In numerical work with $m>r$, contour calculations are therefore performed on the same rank-$r$ projected pencil used for generalized eigenvalue recovery.

Let $\Gamma$ be a positively oriented simple closed $C^1$ contour containing no node on its boundary.  Write $\operatorname{int}\Gamma$ for its interior.

\begin{proposition}[Contour count for the Hankel pencil]\label{thm:argument-count}
The number of nodes in $\operatorname{int}\Gamma$, counted with multiplicity, is
\begin{equation}\label{eq:argument-count}
N_q(\Gamma)
=\frac{1}{2\pi i}\int_\Gamma\frac{\Psi_q'(z)}{\Psi_q(z)}\,dz
=-\frac{1}{2\pi i}\int_\Gamma
\operatorname{tr}\!\left(A_q(z)^{-1}H_{r,0}^{(q)}\right)\,dz.
\end{equation}
\end{proposition}

\begin{proof}
The first identity is the argument principle applied to \eqref{eq:pencil-characteristic-factor} \cite{Ahlfors1979}.  Since $A_q'(z)=-H_{r,0}^{(q)}$, Jacobi's formula gives
$$
\frac{\Psi_q'(z)}{\Psi_q(z)}
=\operatorname{tr}\!\left(A_q(z)^{-1}A_q'(z)\right)
=-\operatorname{tr}\!\left(A_q(z)^{-1}H_{r,0}^{(q)}\right).
$$
\end{proof}

Suppose now that the two minimal blocks are perturbed according to
\begin{equation}\label{eq:noisy-minimal-blocks}
\widetilde H_{r,s}^{(q)}
=H_{r,s}^{(q)}+\Delta_{r,s}^{(q)},
\qquad s=0,1,
\end{equation}
and set
\begin{equation}\label{eq:noisy-characteristic}
\widetilde A_q(z)
:=\widetilde H_{r,1}^{(q)}-z\widetilde H_{r,0}^{(q)},
\qquad
\widetilde\Psi_q(z):=\det\widetilde A_q(z).
\end{equation}
For a contour disjoint from the exact nodes, define
\begin{equation}\label{eq:contour-margins}
\mu_q(\Gamma):=\min_{z\in\Gamma}|\Psi_q(z)|,
\qquad
\zeta_q(\Gamma):=\min_{z\in\Gamma}\sigma_{\min}(A_q(z)).
\end{equation}
Both quantities are positive.

\begin{theorem}[Rouch\'{e} stability of the pencil count]\label{thm:rouche-count}
If
\begin{equation}\label{eq:rouche-direct}
\sup_{z\in\Gamma}|\widetilde\Psi_q(z)-\Psi_q(z)|
<\mu_q(\Gamma),
\end{equation}
then $\widetilde\Psi_q$ and $\Psi_q$ have the same number of zeros in $\operatorname{int}\Gamma$, counted with multiplicity.  It is sufficient that
\begin{equation}\label{eq:rouche-resolvent}
\sup_{z\in\Gamma}
\left\|A_q(z)^{-1}
\bigl(\Delta_{r,1}^{(q)}-z\Delta_{r,0}^{(q)}\bigr)\right\|_2
<2^{1/r}-1.
\end{equation}
In particular, with $R_\Gamma=\max_{z\in\Gamma}|z|$, the condition
\begin{equation}\label{eq:rouche-smin}
\|\Delta_{r,1}^{(q)}\|_2
+R_\Gamma\|\Delta_{r,0}^{(q)}\|_2
<\bigl(2^{1/r}-1\bigr)\zeta_q(\Gamma)
\end{equation}
is sufficient.
\end{theorem}

\begin{proof}
Condition \eqref{eq:rouche-direct} is Rouch\'{e}'s criterion \cite{Ahlfors1979}.  For the resolvent condition, write
$$
\widetilde A_q(z)=A_q(z)(I+X(z)),
\qquad
X(z)=A_q(z)^{-1}
\bigl(\Delta_{r,1}^{(q)}-z\Delta_{r,0}^{(q)}\bigr).
$$
Then
$$
\widetilde\Psi_q(z)-\Psi_q(z)
=\Psi_q(z)\bigl(\det(I+X(z))-1\bigr).
$$
Expansion into principal minors gives
$$
|\det(I+X)-1|
\le(1+\|X\|_2)^r-1.
$$
The right-hand side is smaller than one under \eqref{eq:rouche-resolvent}.  Finally,
$$
\|X(z)\|_2
\le
\frac{\|\Delta_{r,1}^{(q)}\|_2
+|z|\|\Delta_{r,0}^{(q)}\|_2}
{\sigma_{\min}(A_q(z))},
$$
which yields \eqref{eq:rouche-smin}.
\end{proof}

If the samples satisfy
$$
\max_{0\le n\le2r-1}|\eta_n^{(q)}|\le\eps,
$$
then both perturbation blocks in \eqref{eq:noisy-minimal-blocks} have norm at most $r\eps$.  Hence
\begin{equation}\label{eq:rouche-sample-bound}
r\eps(1+R_\Gamma)
<\bigl(2^{1/r}-1\bigr)\zeta_q(\Gamma)
\end{equation}
is sufficient for contour-count stability.  The quantity $\zeta_q(\Gamma)$ in this condition depends on the exact pencil.  For computation it is preferable to reverse the perturbation argument and use the observed pencil as the reference.

\begin{corollary}[A posteriori contour certificate]\label{cor:rouche-aposteriori}
Define the observable boundary margin
\begin{equation}\label{eq:noisy-zeta}
\widetilde\zeta_q(\Gamma)
:=\min_{z\in\Gamma}
\sigma_{\min}\bigl(\widetilde A_q(z)\bigr).
\end{equation}
If $\max_{0\le n\le2r-1}|\eta_n^{(q)}|\le\eps$ and
\begin{equation}\label{eq:rouche-aposteriori}
r\eps(1+R_\Gamma)
<\bigl(2^{1/r}-1\bigr)\widetilde\zeta_q(\Gamma),
\end{equation}
then the exact characteristic $\Psi_q$ and the observed characteristic $\widetilde\Psi_q$ have the same number of zeros in $\operatorname{int}\Gamma$, counted with multiplicity.
\end{corollary}

\begin{proof}
Write
$$
A_q(z)=\widetilde A_q(z)-
\bigl(\Delta_{r,1}^{(q)}-z\Delta_{r,0}^{(q)}\bigr)
$$
and repeat the determinant perturbation argument in the proof of Theorem~\ref{thm:rouche-count}, now with $\widetilde A_q(z)$ as the reference matrix.  Since
$$
\left\|\widetilde A_q(z)^{-1}
\bigl(\Delta_{r,1}^{(q)}-z\Delta_{r,0}^{(q)}\bigr)\right\|_2
\le
\frac{r\eps(1+R_\Gamma)}{\widetilde\zeta_q(\Gamma)},
$$
condition \eqref{eq:rouche-aposteriori} gives the required Rouch\'{e} inequality.
\end{proof}

For numerical certification, the continuous boundary minimum in \eqref{eq:noisy-zeta} can be bounded from finitely many samples.  Let $z_1,\ldots,z_M\in\Gamma$ and define the fill distance
$$
h_M:=\sup_{z\in\Gamma}\min_{1\le k\le M}|z-z_k|.
$$

\begin{lemma}\label{lem:finite-boundary-control}
Set
\begin{equation}\label{eq:discrete-zeta-lower}
\underline\zeta_{q,M}(\Gamma)
:=
\min_{1\le k\le M}
\sigma_{\min}\bigl(\widetilde A_q(z_k)\bigr)
-h_M\|\widetilde H_{r,0}^{(q)}\|_2.
\end{equation}
Then
$$
\widetilde\zeta_q(\Gamma)
\ge
\underline\zeta_{q,M}(\Gamma).
$$
Consequently, if $\underline\zeta_{q,M}(\Gamma)>0$ and
\begin{equation}\label{eq:finite-grid-certificate}
r\eps(1+R_\Gamma)
<
\bigl(2^{1/r}-1\bigr)\underline\zeta_{q,M}(\Gamma),
\end{equation}
the exact and observed characteristics have the same zero count in $\operatorname{int}\Gamma$.

Moreover, let $\gamma$ be a $C^1$ subarc of $\Gamma$ of length $L_\gamma$ on which
$\sigma_{\min}(\widetilde A_q(z))\ge\zeta_\gamma>0$.  Then
\begin{equation}\label{eq:phase-variation-bound}
\left|\Delta_\gamma\arg\widetilde\Psi_q\right|
\le
\frac{r\|\widetilde H_{r,0}^{(q)}\|_2}{\zeta_\gamma}\,L_\gamma.
\end{equation}
\end{lemma}

\begin{proof}
The smallest singular value is $1$-Lipschitz with respect to the operator norm.  Since
$$
\widetilde A_q(z)-\widetilde A_q(z_k)
=-(z-z_k)\widetilde H_{r,0}^{(q)},
$$
taking the nearest boundary sample gives \eqref{eq:discrete-zeta-lower}.  The certificate then follows from Corollary~\ref{cor:rouche-aposteriori}.  For the phase bound, Jacobi's formula gives
$$
\frac{d}{dz}\log\widetilde\Psi_q(z)
=
-\operatorname{tr}\!\left(
\widetilde A_q(z)^{-1}\widetilde H_{r,0}^{(q)}
\right).
$$
Integration along $\gamma$, together with
$|\operatorname{tr}B|\le r\|B\|_2$, yields \eqref{eq:phase-variation-bound}.
\end{proof}

The same result gives local information.  Let
$$
\delta_\ell(q):=\min_{k\ne\ell}|\rho_\ell(q)-\rho_k(q)|
$$
and choose $0<d<\delta_\ell(q)$.  On
$$
\Gamma_\ell(d)=\{z:|z-\rho_\ell(q)|=d\},
$$
formula \eqref{eq:pencil-characteristic-factor} yields
\begin{equation}\label{eq:local-det-margin}
\mu_q(\Gamma_\ell(d))
\ge
|\det H_{r,0}^{(q)}|\,d
\prod_{k\ne\ell}
\bigl(|\rho_\ell(q)-\rho_k(q)|-d\bigr).
\end{equation}

\begin{corollary}\label{cor:single-node-rouche}
If Theorem~\ref{thm:rouche-count} applies on $\Gamma_\ell(d)$, then $\widetilde\Psi_q$ has exactly one zero $\widehat\rho_\ell$ in $|z-\rho_\ell(q)|<d$.  Thus
$$
|\widehat\rho_\ell-\rho_\ell(q)|<d.
$$
If also $d<|\rho_\ell(q)|$ and $q^{\widehat\alpha_\ell}=\widehat\rho_\ell$, then for some $k\in\mathbb Z$,
\begin{equation}\label{eq:local-alpha-rouche}
\left|
\widehat\alpha_\ell-\alpha_\ell-
\frac{2\pi i k}{\log q}
\right|
\le
\frac{d}{|\log q|\bigl(|\rho_\ell(q)|-d\bigr)}.
\end{equation}
\end{corollary}

\begin{proof}
The exact disk contains only $\rho_\ell(q)$, and the Rouch\'{e} condition preserves its zero count.  If $d<|\rho_\ell(q)|$, the disk avoids the origin and carries an analytic logarithm.  Integrating $1/z$ along the segment from $\rho_\ell(q)$ to $\widehat\rho_\ell$ gives the bound; different logarithm branches differ by $2\pi ik$.
\end{proof}

A contour may also surround a whole cluster.  If it encloses precisely the nodes indexed by a set $S$, Theorem~\ref{thm:rouche-count} preserves the count $|S|$ even when no disjoint family of single-node contours satisfies the perturbation criterion.  The contour count therefore remains informative after pointwise localization has broken down.

Finally, the determinant margin inherits the singular behavior of \eqref{eq:q-one-det}.  Let $\Gamma$ be fixed, contain $1$ in its interior, and put
$$
d_\Gamma=\min_{z\in\Gamma}|1-z|>0.
$$

\begin{corollary}[Contour-margin degeneration as $q\to1$]\label{cor:rouche-margin-q1}
Under the assumptions of Corollary~\ref{cor:q-one-det},
\begin{equation}\label{eq:rouche-margin-q1}
\mu_q(\Gamma)
=C_\Gamma|\log q|^{r(r-1)}(1+o(1)),
\qquad q\to1,
\end{equation}
where
\begin{equation}\label{eq:rouche-margin-constant}
C_\Gamma
=\left|\prod_{\ell=1}^{r}w_\ell\right|
\prod_{k<\ell}|\alpha_\ell-\alpha_k|^2\,d_\Gamma^r.
\end{equation}
\end{corollary}

\begin{proof}
By \eqref{eq:pencil-characteristic-factor},
$$
\mu_q(\Gamma)
=|\det H_{r,0}^{(q)}|
\min_{z\in\Gamma}
\prod_{\ell=1}^{r}|\rho_\ell(q)-z|.
$$
Since every $\rho_\ell(q)\to1$, the product converges uniformly on $\Gamma$ to $|1-z|^r$, whose minimum is $d_\Gamma^r$.  Combining this with Corollary~\ref{cor:q-one-det} proves the claim.
\end{proof}

For real exponents, the estimates above quantify the conditioning benefit of an auxiliary scale.  For complex exponents, the scale must additionally resolve the aliasing lattice of Section~\ref{sec:identifiability}.  The reconstruction below uses both effects.
\section{Dual-channel recovery algorithm and numerical experiments}
\label{sec:algorithm}

The same Hankel data are used in two different ways.  After rank reduction, a projected pencil provides point estimates of the scale-dependent nodes.  Independently, the determinant characteristic of a minimal pencil is evaluated on prescribed contours; its winding number gives a regional node count without first locating the zeros.  The two outputs are compared only after they have been formed.

Let
\begin{equation}\label{eq:data-alg}
\widetilde y_n^{(j)}
=\sum_{\ell=1}^{r}w_\ell(q_j^{\alpha_\ell})^n+\eta_n^{(j)},
\qquad
n=0,\ldots,N_j-1,\quad j=1,\ldots,s.
\end{equation}
For each scale choose $m_j$ with $2m_j\le N_j$ and form
\begin{equation}\label{eq:hankel-alg}
\widetilde H_{m_j,0}^{(q_j)}
=\left(\widetilde y_{a+b}^{(j)}\right)_{a,b=0}^{m_j-1},
\qquad
\widetilde H_{m_j,1}^{(q_j)}
=\left(\widetilde y_{a+b+1}^{(j)}\right)_{a,b=0}^{m_j-1}.
\end{equation}
Rank selection is carried out scale by scale because the Vandermonde subspaces depend on $q_j$.  Once a common order $\widehat r$ has been selected, the point channel may use a larger Hankel window, whereas the contour channel uses the minimal $\widehat r\times\widehat r$ blocks formed from the first $2\widehat r$ samples.

\begin{algorithm}[htbp]
\caption{Dual-channel multiscale Hankel recovery}
\label{alg:dual-channel}
\begin{algorithmic}[1]
\Require Scales $q_1,\ldots,q_s$, common base point $x_0$, data $\{\widetilde y_n^{(j)}\}$, admissible exponent set $\Omega$, node-search regions $\mathcal R_j$, and, when available, deterministic noise bounds $\eps_j$.
\Ensure Model order $\widehat r$, point estimates, contour cells, resolved exponents and amplitudes, and any unresolved spectral clusters.
\For{$j=1,\ldots,s$}
  \State Form the consecutive Hankel blocks in \eqref{eq:hankel-alg}.
  \State Estimate the scale-wise numerical rank from the singular-value gap.
\EndFor
\State Combine the scale-wise ranks by the prescribed consensus rule to obtain $\widehat r$.
\For{$j=1,\ldots,s$}
  \State \textbf{Point channel:} form a rank-$\widehat r$ projected pencil, compute point estimates $\widehat\rho_{\ell,j}$, and recover $\widehat w_{\ell,j}$ by Vandermonde least squares.
  \State \textbf{Contour channel:} form the minimal observed pencil and apply Algorithm~\ref{alg:contour-localization} on $\mathcal R_j$.
  \State Associate point estimates with contour cells when possible.  A certified singleton cell establishes that exactly one true node lies in that region; a certified cell of count greater than one is retained as an unresolved cluster.
\EndFor
\State Match components that are resolved on at least one scale by amplitude proximity and the compatibility cost $d_\Omega$; use auxiliary scales to split unresolved reference-scale clusters whenever possible.
\For{each component resolved and matched across the available scales}
  \State Compute
  $$
  \widehat\alpha_\ell
  =\argmin_{\alpha\in\Omega}
  \sum_{j=1}^{s}\vartheta_j
  |e^{\alpha\log q_j}-\widehat\rho_{\ell,j}|^2.
  $$
\EndFor
\State Recover the corresponding effective amplitudes and set $\widehat a_\ell=\widehat w_\ell x_0^{-\widehat\alpha_\ell}$.
\State Return any cluster unresolved at all scales as a regional spectral object, without assigning artificial pointwise exponents.
\State Optionally refine the resolved components by variable projection or damped Gauss--Newton iteration for the original multiscale model.
\end{algorithmic}
\end{algorithm}

For ordered boundary samples $z_0,\ldots,z_{M-1}$, with $z_M=z_0$, the direct winding approximation is
\begin{equation}\label{eq:discrete-winding}
\widehat N_M(\Gamma)
=\frac{1}{2\pi}
\sum_{k=0}^{M-1}
\operatorname{Arg}
\left(
\frac{\widetilde\Psi_q(z_{k+1})}
     {\widetilde\Psi_q(z_k)}
\right).
\end{equation}
The determinant phase is evaluated through a complex log-determinant.  Boundary sampling is refined until the integer winding is stable.  For rigorous phase unwrapping, Lemma~\ref{lem:finite-boundary-control} is applied segmentwise: if a segment has length $L_k$ and a certified lower bound $\underline\zeta_k$ for $\sigma_{\min}(\widetilde A_q)$ on that segment, then refinement continues until
\begin{equation}\label{eq:phase-segment-criterion}
\frac{\widehat r\,\|\widetilde H_{\widehat r,0}^{(q)}\|_2}
     {\underline\zeta_k}L_k
<\frac{\pi}{2}
\end{equation}
for every segment.  The same boundary samples yield the finite-grid lower bound \eqref{eq:discrete-zeta-lower} used in the Rouch\'{e} test.

\begin{algorithm}[htbp]
\caption{Adaptive contour localization and Rouch\'{e} certification}
\label{alg:contour-localization}
\begin{algorithmic}[1]
\Require Minimal observed pencil $\widetilde A_q(z)$, initial rectangle $\mathcal R$, tolerance $\tau$, and, when available, a noise bound $\eps$.
\Ensure Cells $(\mathcal R_\nu,N_\nu,c_\nu)$, where $N_\nu$ is the winding count and $c_\nu$ records whether the count is certified for the exact spectrum.
\State Sample $\partial\mathcal R$ adaptively and compute the direct winding count from \eqref{eq:discrete-winding}; refine until the winding is stable and the segmentwise bound \eqref{eq:phase-segment-criterion} is satisfied.
\State Compute the finite-grid boundary lower bound \eqref{eq:discrete-zeta-lower}; when $\eps$ is known, test \eqref{eq:finite-grid-certificate}.
\If{the count is zero}
  \State discard $\mathcal R$.
\ElsIf{$\operatorname{diam}(\mathcal R)\le\tau$}
  \State return the cell with its count and certification flag.
\Else
  \State choose vertical and horizontal split lines away from small values of $\sigma_{\min}(\widetilde A_q(z))$.
  \State compute winding counts and certification flags on the nonempty child rectangles.
  \If{the certified child counts add to the certified parent count}
    \State recurse on the nonempty children.
  \Else
    \State retain the parent as a certified unresolved cluster if its certificate holds; otherwise retain it as an uncertified cell.
  \EndIf
\EndIf
\end{algorithmic}
\end{algorithm}

In the point channel, singular-value truncation is followed by a reduced pencil or ESPRIT solve rather than direct root finding of a high-degree annihilating polynomial \cite{Liao2016,KatzDiabBatenkov2024}.  The contour channel solves a different problem: it returns regional counts and, when a deterministic noise bound is available, certifies those counts against the exact spectrum.

\paragraph{Cross-scale matching and branch resolution.}
For two noisy scales we use the cost
\begin{equation}\label{eq:matching-cost-noisy}
C_{k\ell}
=\tau_w|\widehat w_{k,1}-\widehat w_{\ell,2}|^2
+\tau_\alpha d_\Omega(\widehat\rho_{k,1},\widehat\rho_{\ell,2})^2,
\end{equation}
with $\tau_w,\tau_\alpha>0$.  A contour cell of count greater than one is not assigned point labels until an auxiliary scale resolves its components.  For an individually matched node, the single-scale branch candidates are
\begin{equation}\label{eq:branch-candidates}
\alpha_{\ell,j}^{(k)}
=\frac{\Log\widehat\rho_{\ell,j}+2\pi ik}{\log q_j},
\qquad k\in\mathbb Z.
\end{equation}
A bounded admissible set $\Omega$ leaves finitely many candidates, and the joint least-squares step selects the mutually consistent branch.

\paragraph{Joint refinement.}
For
$$
\theta=(\alpha_1,\ldots,\alpha_r,w_1,\ldots,w_r),
$$
the multiscale objective is
\begin{equation}\label{eq:joint-NLS}
\mathcal J(\theta)
=\sum_{j=1}^{s}\sum_{n=0}^{N_j-1}
\left|
\widetilde y_n^{(j)}
-\sum_{\ell=1}^{r}w_\ell e^{n\alpha_\ell\log q_j}
\right|^2.
\end{equation}
For fixed exponents the amplitudes enter linearly, so variable projection removes them from the nonlinear search.  The point channel supplies the initialization; the contour cells indicate whether that initialization represents isolated nodes or only an unresolved group.

The numerical experiments focus on spectral recovery and scale effects rather than on comparing rank-selection heuristics.  The projected pencils therefore use the known value $r$; Experiment~I separately displays the singular-value gap that reveals this order in the noiseless case.

\subsection{Numerical experiment I: exact finite-rank recovery}

The first experiment checks the finite-rank factorization and the
projected-pencil implementation in a noiseless single-scale setting.
We take
$$
\alpha=(0.45,1.35,2.80),
\qquad
a=(1.20,-0.70,0.50),
$$
with $x_0=1.3$, $q=0.72$, $N=18$ samples, and a $7\times7$
Hankel matrix.  The three leading singular values are
$$
3.9050,\quad
2.9165\times10^{-1},\quad
4.0379\times10^{-2},
$$
while the remaining singular values are at machine precision.
As shown in Figure~\ref{fig:exp1-rank}, this produces a sharp
singular-value gap after the third mode and identifies the numerical
Hankel rank as $r=3$.  The projected pencil then recovers
$$
\widehat\alpha=(0.45,1.35,2.80)
$$
with maximum exponent error $9.55\times10^{-14}$.

\begin{figure}[htbp]
\centering
\includegraphics[width=0.68\textwidth]{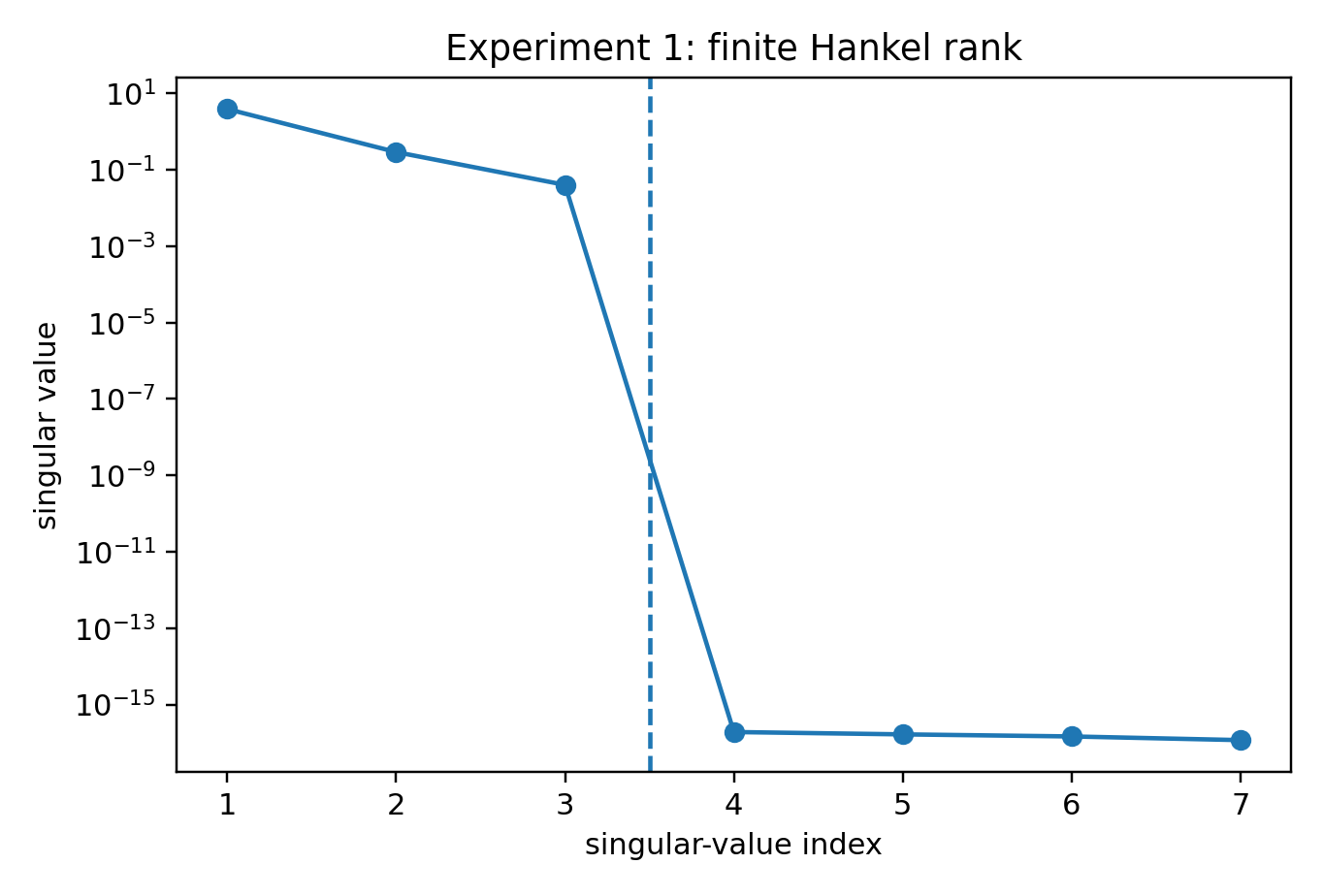}
\caption{Singular values of the $7\times7$ geometric Hankel matrix
in the exact three-mode experiment.  The sharp drop after the third
singular value separates the three-dimensional signal subspace from
the numerical null space; the dashed line marks the corresponding
rank cutoff.}
\label{fig:exp1-rank}
\end{figure}

\subsection{Numerical experiment II: degeneration as \texorpdfstring{$q\to1$}{q to 1} and scale selection}

To test Corollary~\ref{cor:q-one-det}, we choose four real exponents
$$
\alpha=(0.4,1.2,2.3,3.7)
$$
and effective amplitudes
$$
w=(1.1,0.9,-0.6,0.4).
$$
The predicted degeneration exponent is $r(r-1)=12$.  Because ordinary
double precision suffers severe cancellation in this regime, the determinants
are evaluated with 80-digit arithmetic for $q\in[0.999,0.99999]$.
A least-squares fit of $\log|\det H_{r,0}^{(q)}|$ against $\log|\log q|$
gives the slope
$$
11.994624,
$$
which differs from the theoretical value by approximately
$5.4\times10^{-3}$.  As shown in Figure~\ref{fig:exp2-det}, the computed
determinants lie nearly on a straight line in log--log coordinates, in
agreement with the predicted power law.

\begin{figure}[htbp]
\centering
\includegraphics[width=0.68\textwidth]{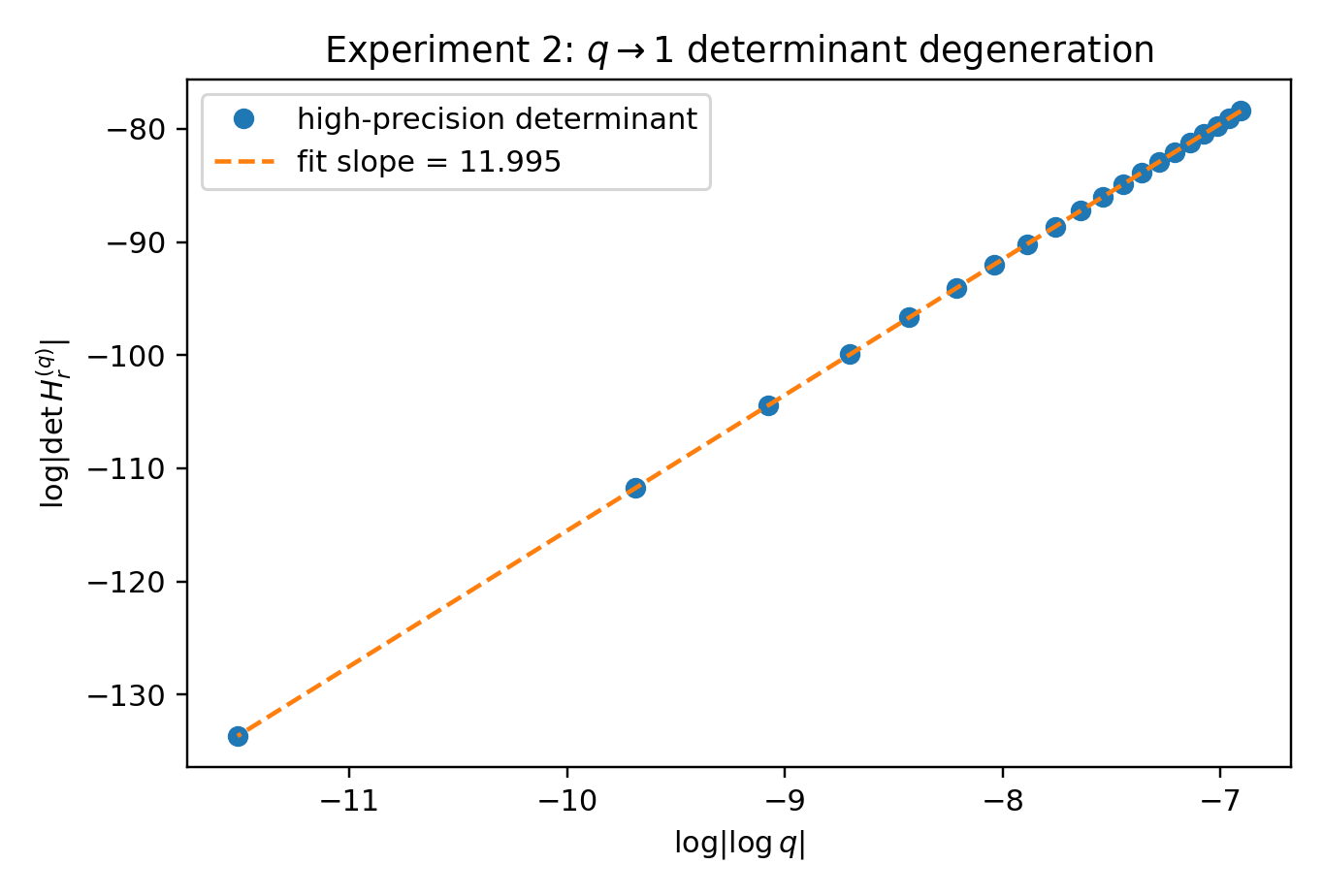}
\caption{Log--log verification of the determinant degeneration law as
$q\to1$.  The fitted slope is $11.995$, close to the theoretical value
$r(r-1)=12$.}
\label{fig:exp2-det}
\end{figure}

For comparison with the determinant calculation, Figure~\ref{fig:exp2-scale}
displays the first-order scale surrogate $|\log q|q^B$ with $B=4$.
The analytic maximizer \eqref{eq:q-star} is
$q_*=e^{-1/4}=0.778801$, while the grid maximizer is $0.779419$.

\begin{figure}[htbp]
\centering
\includegraphics[width=0.68\textwidth]{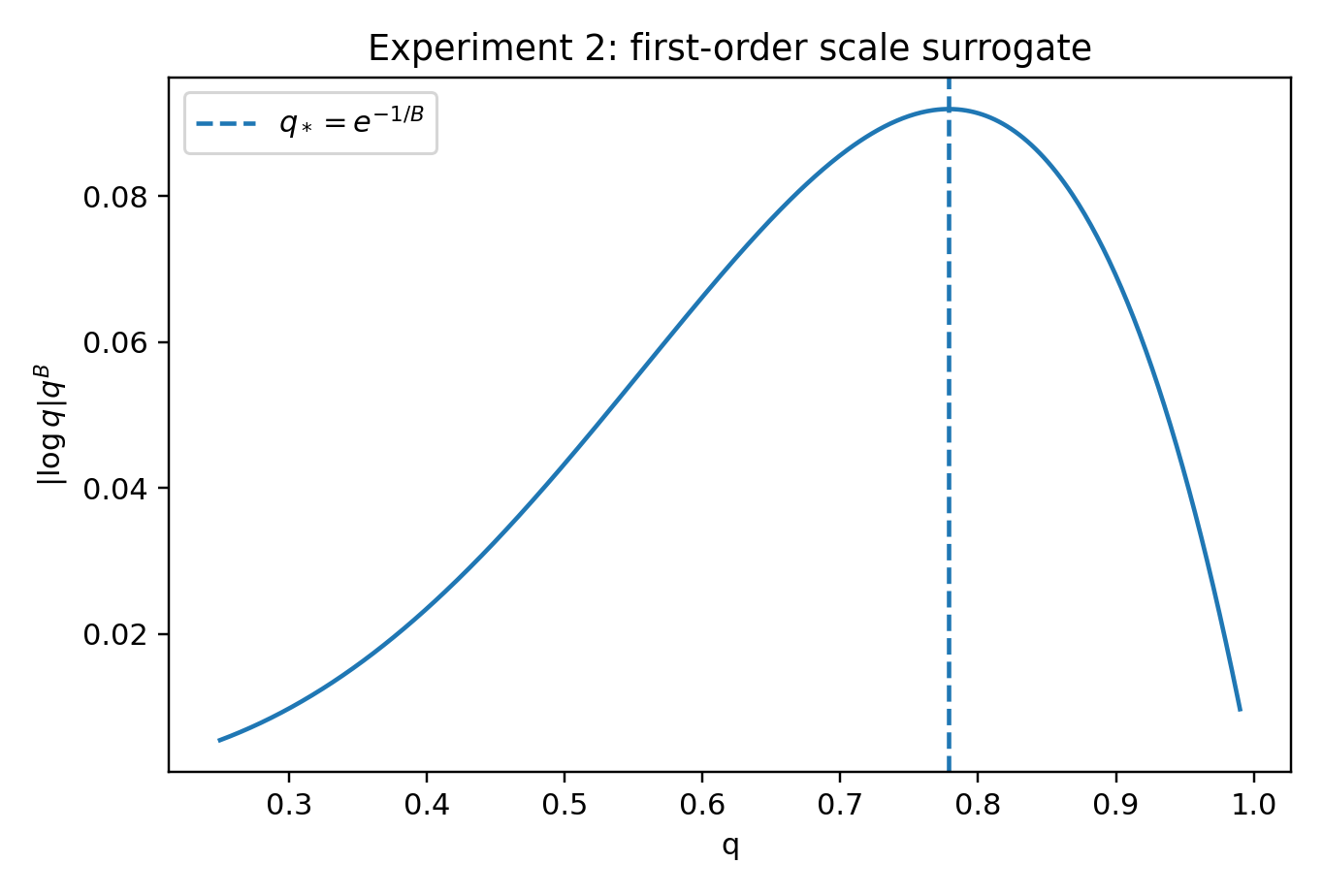}
\caption{The first-order scale surrogate $|\log q|q^B$ for $B=4$.
The dashed line marks the analytic maximizer $q_*=e^{-1/B}$.}
\label{fig:exp2-scale}
\end{figure}

\subsection{Numerical experiment III: complex aliasing and two-scale de-aliasing}

The third experiment tests the multiscale aliasing statement in
Theorem~\ref{thm:multi-alias}.  Let
$$
h_1=-0.8,
\qquad
h_2=-0.8\sqrt{2},
\qquad
q_j=e^{h_j},
$$
so that
$$
\frac{h_1}{h_2}=\frac{1}{\sqrt2}\notin\mathbb Q.
$$
Starting from
$$
\alpha=0.7+1.4i,
$$
we form the exact first-scale alias
$$
\alpha_{\rm alias}
=
\alpha+\frac{2\pi i}{h_1}
=
0.7-6.4539816339\,i.
$$
On the first grid the corresponding nodes differ by only
$2.48\times10^{-16}$, so the two exponents are numerically
indistinguishable at that scale.  On the second grid, however, the
difference is $0.8732$, which removes the alias.  Enumerating the
logarithmic branches generated by the first scale and testing them
against the second selects the true branch with error
$1.11\times10^{-16}$.  As shown in
Figure~\ref{fig:exp3-alias}, the second-scale compatibility residual
drops to machine precision at the correct branch index $k=0$, while
the competing branches remain separated by an $O(1)$ residual.

\begin{figure}[htbp]
\centering
\includegraphics[width=0.68\textwidth]{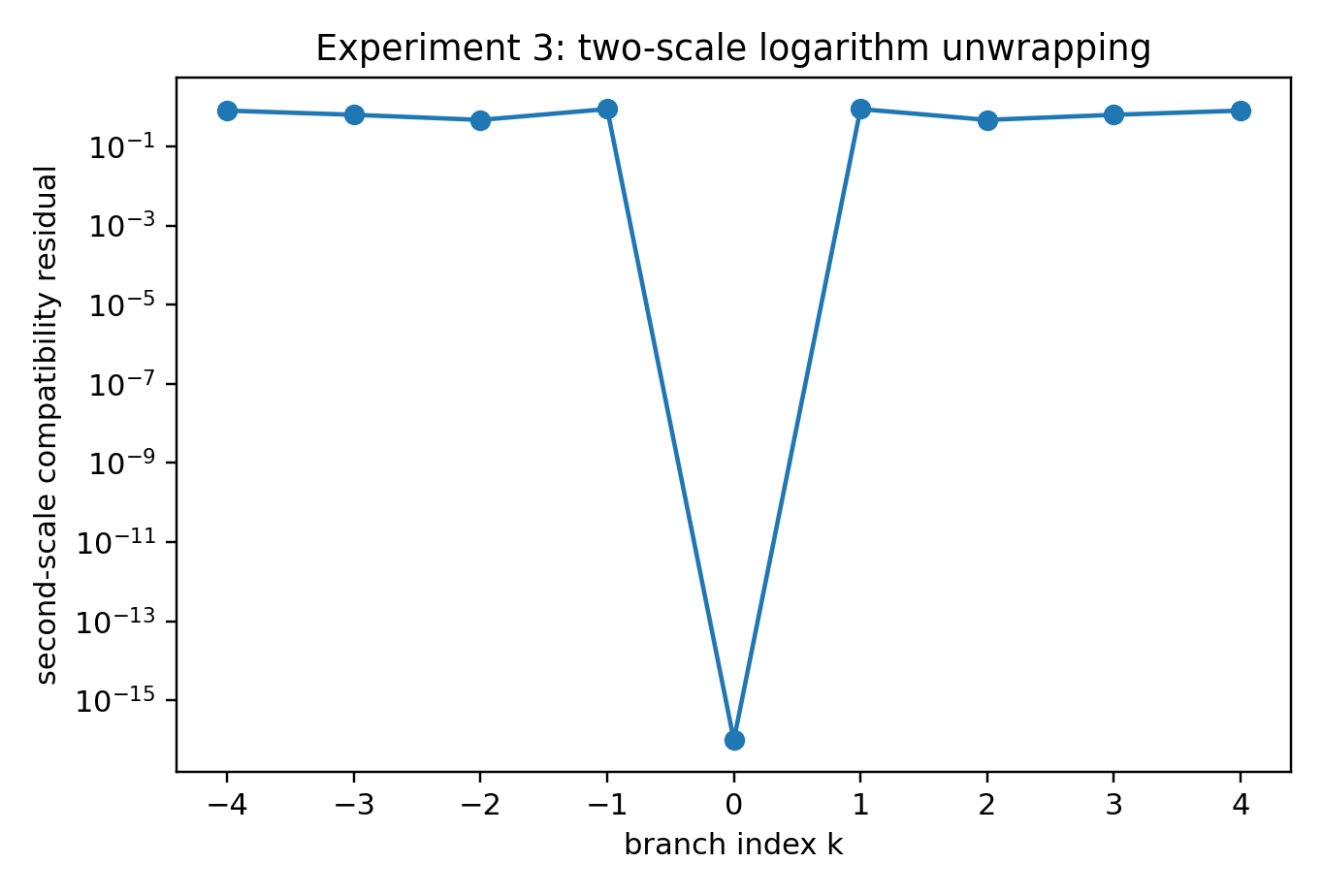}
\caption{Second-scale compatibility residual for the logarithmic
branches generated by the first scale.  The auxiliary scale, whose
logarithmic step is irrationally related to the first, singles out the
correct branch $k=0$.}
\label{fig:exp3-alias}
\end{figure}

\subsection{Numerical experiment IV: clustered spectra under noise}

The fourth experiment examines a clustered real spectrum for which a
sampling ratio close to one produces a severely ill-conditioned Hankel
matrix.  We take
$$
\alpha=(1.80,1.95,2.10),
\qquad
w=(1.0,-0.8,0.6),
$$
with $N=18$ samples and a $7\times7$ Hankel matrix.  Two scales are
compared:
$$
q_{\rm bad}=0.97,
\qquad
q_{\rm good}=0.65.
$$
In the noiseless case, the third singular values are
$$
\sigma_3(H_{\rm bad})=1.103\times10^{-8},
\qquad
\sigma_3(H_{\rm good})=2.922\times10^{-6}.
$$
Thus the better separated scale increases the singular value governing
the rank-three signal subspace by more than two orders of magnitude.

For each relative noise level $\delta$, we perform $150$ independent
trials.  On scale $q_j$, independent real Gaussian noise is added with
standard deviation
$$
\sigma_j
=
\delta\,\frac{\|y^{(j)}\|_2}{\sqrt{N}},
$$
so that $\delta$ measures the noise level relative to the root-mean-square
signal magnitude on that grid.  The two single-scale reconstructions use
the rank-three projected pencil.  The joint reconstruction is initialized
from the $q=0.65$ estimate and refined by two-scale variable projection.
Before refinement, the initial exponents are projected componentwise onto
the admissible interval $[1,3]$.  The nonlinear least-squares solver
reports successful termination in all $150$ trials at every noise level.

Table~\ref{tab:clustered} reports the median exponent RMSE.  The scale
$q=0.97$ fails to resolve the clustered spectrum even at the smallest
tested noise level, whereas $q=0.65$ remains stable over the same range.
The joint refinement is essentially indistinguishable from the
well-separated single-scale reconstruction at $\delta=10^{-8}$, is
slightly less accurate at $\delta=10^{-7}$, and improves the median RMSE
from $1.379\times10^{-1}$ to $1.040\times10^{-1}$ at
$\delta=10^{-6}$.

\begin{table}[htbp]
\centering
\caption{Median exponent RMSE over $150$ trials for the clustered-spectrum
experiment.}
\label{tab:clustered}
\begin{tabular}{cccc}
\toprule
Relative noise
& single $q=0.97$
& single $q=0.65$
& two-scale joint \\
\midrule
$10^{-8}$
& $1.322\times10^{1}$
& $3.694\times10^{-3}$
& $3.694\times10^{-3}$ \\
$10^{-7}$
& $1.004\times10^{1}$
& $3.135\times10^{-2}$
& $3.191\times10^{-2}$ \\
$10^{-6}$
& $1.327\times10^{1}$
& $1.379\times10^{-1}$
& $1.040\times10^{-1}$ \\
\bottomrule
\end{tabular}
\end{table}

The first four experiments therefore illustrate two distinct roles of
the sampling scale: removal of logarithmic aliasing for complex exponents
and improvement of numerical resolution for clustered real spectra.

\subsection{Numerical experiment V: dual-channel contour recovery and certification}

The fifth experiment examines the contour channel separately from generalized-eigenvalue recovery.  We use
$$
\alpha=(1.80+0.80i,\ 1.95+0.85i,\ 2.10+0.90i),
\qquad
w=(1.00,\ -0.80+0.15i,\ 0.60-0.10i),
$$
and compare the near-one scale $q=0.90$ with $q=0.65$.  At each scale, the node-search rectangle is obtained from the a priori exponent box
$$
1.5\le\operatorname{Re}\alpha\le2.4,
\qquad
0.5\le\operatorname{Im}\alpha\le1.2,
$$
by bounding its image under $\alpha\mapsto q^\alpha$ and adding a $3\%$ padding.

For exact data, the determinant phase is accumulated directly along the rectangle boundary.  Sixteen points per edge are sufficient for the adaptive winding to stabilize at both scales, and the total phase change is $6\pi$.  Hence the contour channel returns the count three without solving a generalized-eigenvalue problem; Figure~\ref{fig:exp5-direct-winding} shows the accumulated phase.

\begin{figure}[htbp]
\centering
\includegraphics[width=0.70\textwidth]{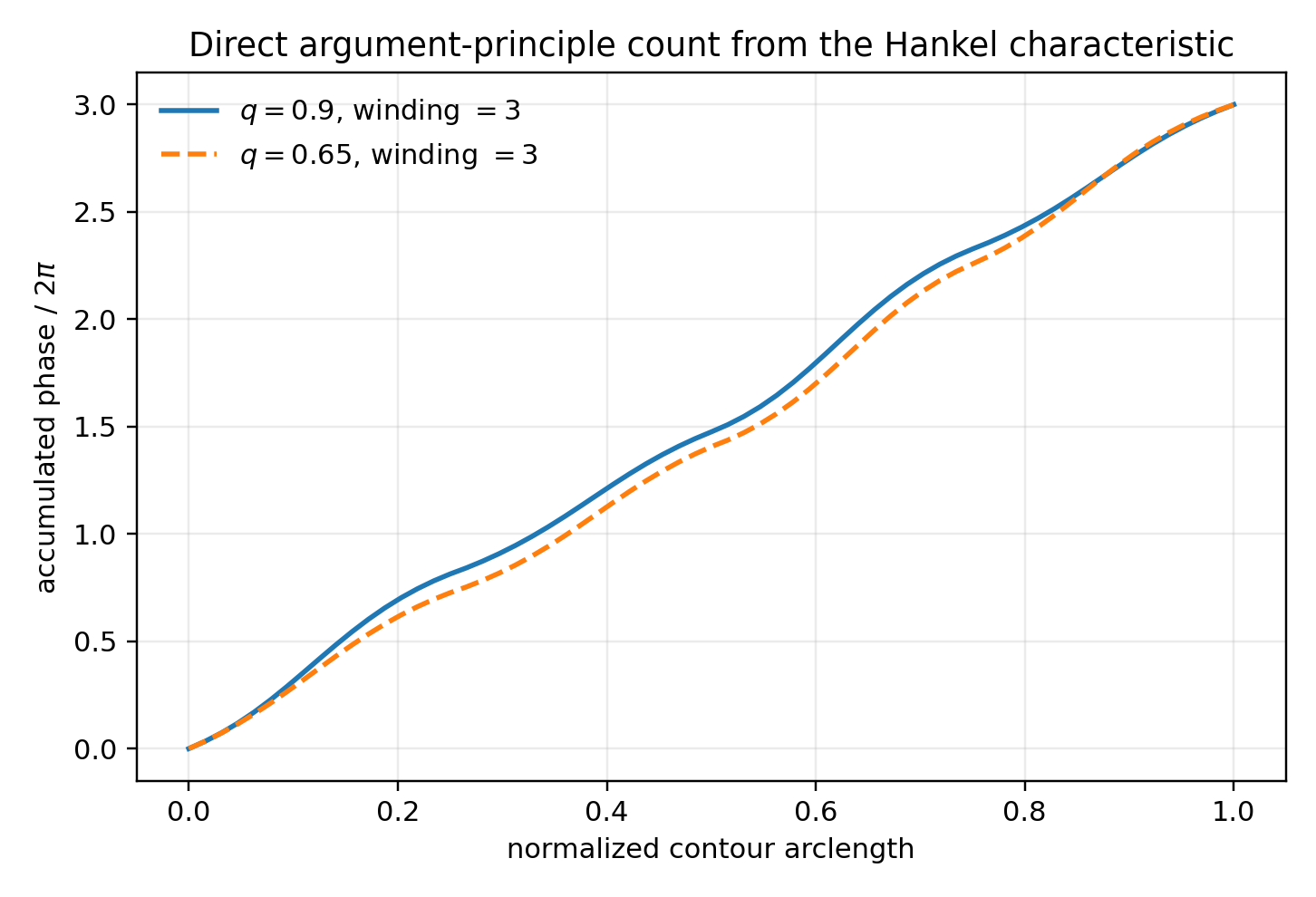}
\caption{Direct argument-principle computation on the two search rectangles.  The accumulated determinant phase completes three turns at both scales.}
\label{fig:exp5-direct-winding}
\end{figure}

For the perturbation study, $N=18$ moments are contaminated by bounded circular complex noise,
$$
|\eta_n|\le\varepsilon,
\qquad
\varepsilon=\delta\frac{\|y\|_2}{\sqrt N}.
$$
The point channel uses a $7\times7$ Hankel window followed by rank-three projection.  The contour channel uses the first six moments to form the minimal $3\times3$ pencil and computes the winding from its determinant characteristic.  A pointwise trial is successful when the maximum matched node error is below $10^{-3}$.  For each $\delta\in\{10^{-13},\ldots,10^{-7}\}$, $150$ independent trials are performed.

The Rouch\'{e} condition involves a continuous minimum along the contour.  In the Monte Carlo experiment this minimum is evaluated numerically on a dense boundary grid; the corresponding curve in Figure~\ref{fig:exp5-dual-channel-noise} should therefore be read as a numerical Rouch\'{e}-test pass rate.  Lemma~\ref{lem:finite-boundary-control} gives a conservative finite-sample lower bound when a fully rigorous boundary certificate is required.

\begin{figure}[htbp]
\centering
\includegraphics[width=0.76\textwidth]{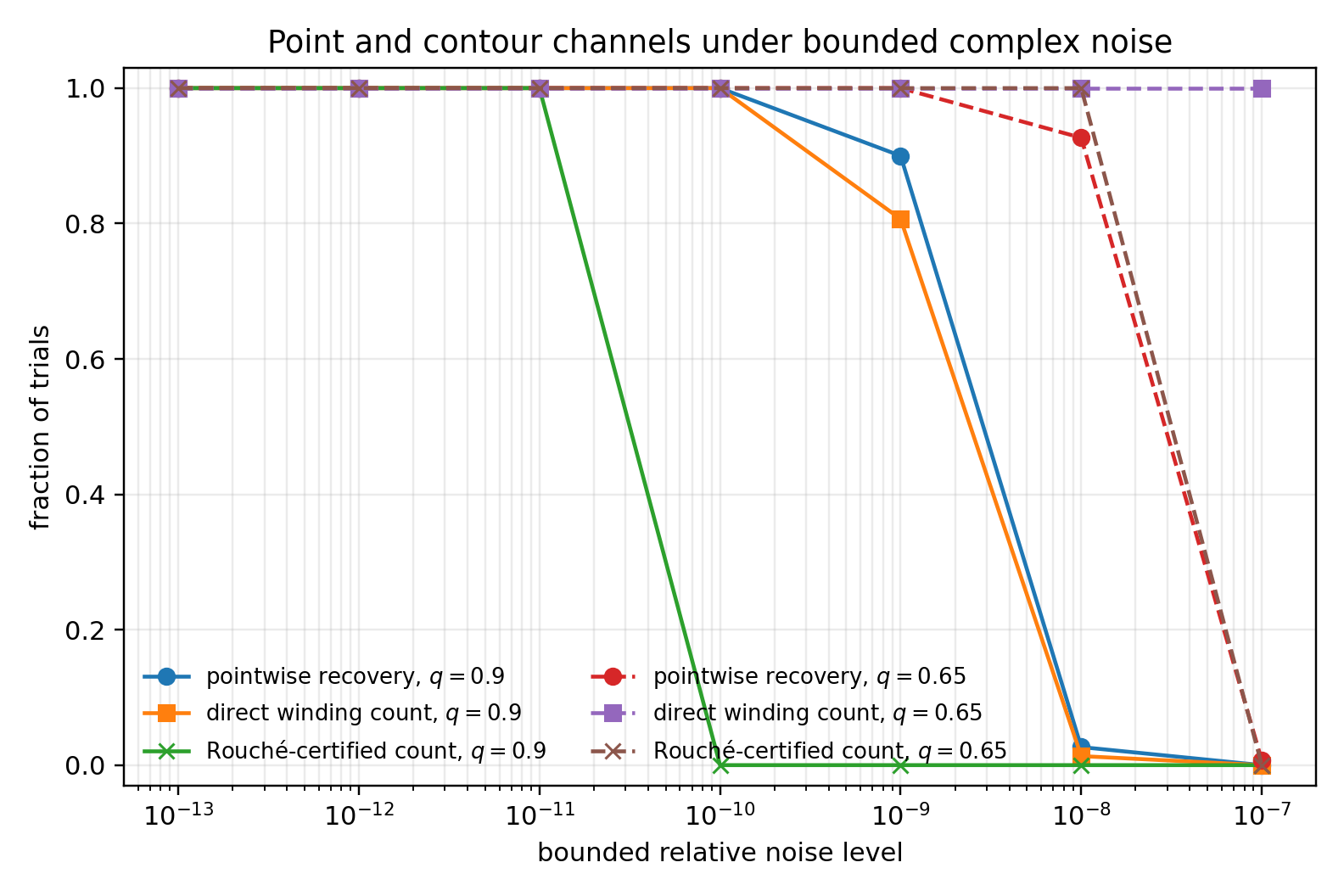}
\caption{Point and contour channels under bounded complex noise.  The winding count is computed directly from the determinant characteristic.  The Rouch\'{e} curve records the numerical pass rate of the data-dependent sufficient condition.}
\label{fig:exp5-dual-channel-noise}
\end{figure}

The distinction between pointwise and regional recovery is most visible at $q=0.65$.  At $\delta=10^{-8}$, the pointwise success rate is $0.927$, whereas the winding count is correct in all $150$ trials and the numerical Rouch\'{e} test also passes throughout.  At $\delta=10^{-7}$, only $0.7\%$ of the point estimates meet the $10^{-3}$ tolerance, but the contour count still returns the correct cardinality in every trial; the Rouch\'{e} test no longer passes at this noise level.  For $q=0.90$, deterioration occurs earlier: at $\delta=10^{-9}$ the pointwise and winding success rates are $0.900$ and $0.807$, respectively.  The experiment also separates empirical correctness from the sufficient certificate.  At $q=0.90$ and $\delta=10^{-10}$, the winding count is correct in every trial although the Rouch\'{e} test already fails.

The same contrast appears in adaptive localization.  At $\delta=10^{-11}$, the subdivision driven by the numerical Rouch\'{e} test retains one count-three cell for $q=0.90$, whereas the $q=0.65$ search region separates into three singleton cells.  The point-channel estimates are overlaid in Figure~\ref{fig:exp5-certified-localization} only for comparison and are not used in the winding computation.

\begin{figure}[htbp]
\centering
\includegraphics[width=0.92\textwidth]{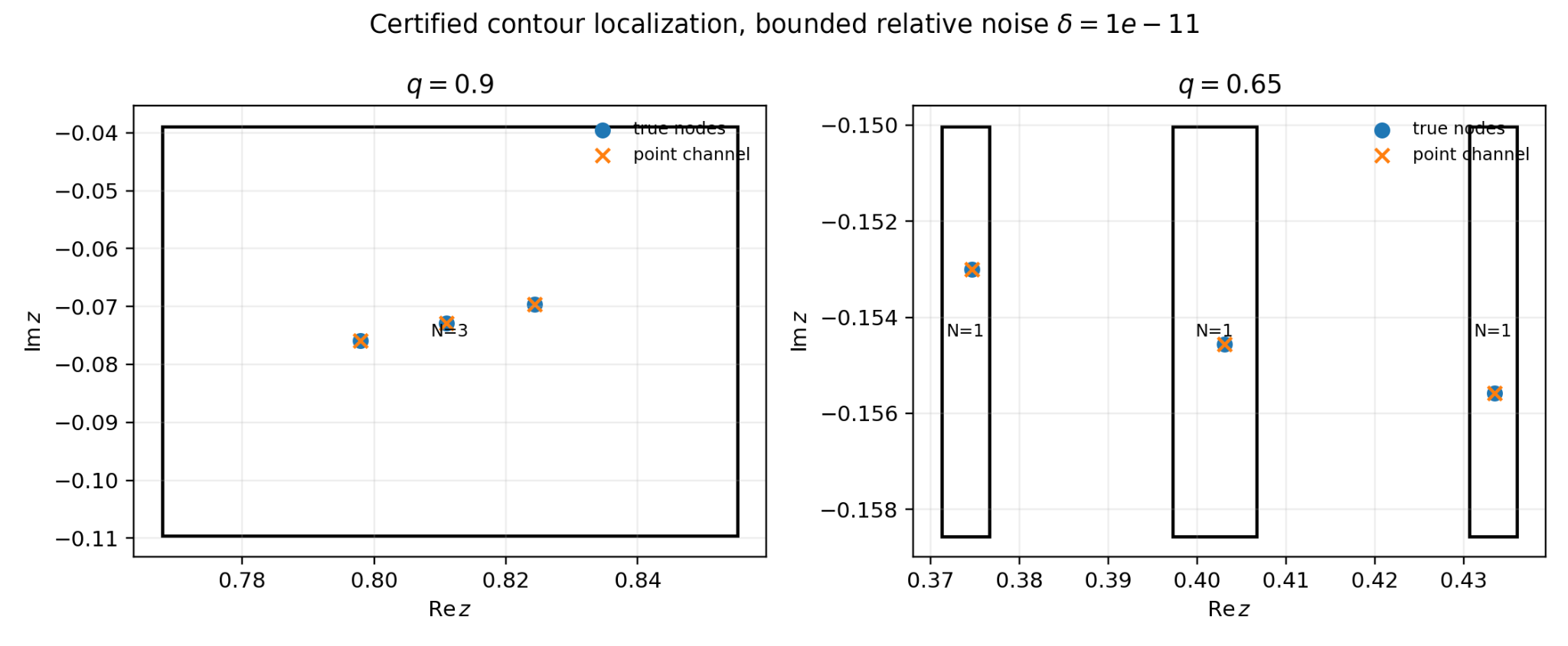}
\caption{Rouch\'{e}-guided localization at $\delta=10^{-11}$.  The near-one scale retains one cell of count three, whereas the auxiliary scale resolves the same spectrum into three singleton cells.  Crosses show the independent point-channel estimates.}
\label{fig:exp5-certified-localization}
\end{figure}

Finally, Corollary~\ref{cor:rouche-margin-q1} is tested on the fixed contour $|z-1|=0.25$.  For $q\in[0.99,0.99999]$, the exact characteristic is evaluated through \eqref{eq:pencil-characteristic-factor} to avoid determinant cancellation.  A least-squares fit of $\log\mu_q(\Gamma)$ against $\log|\log q|$ gives
$$
5.963743,
$$
close to the predicted exponent $r(r-1)=6$.

\begin{figure}[htbp]
\centering
\includegraphics[width=0.68\textwidth]{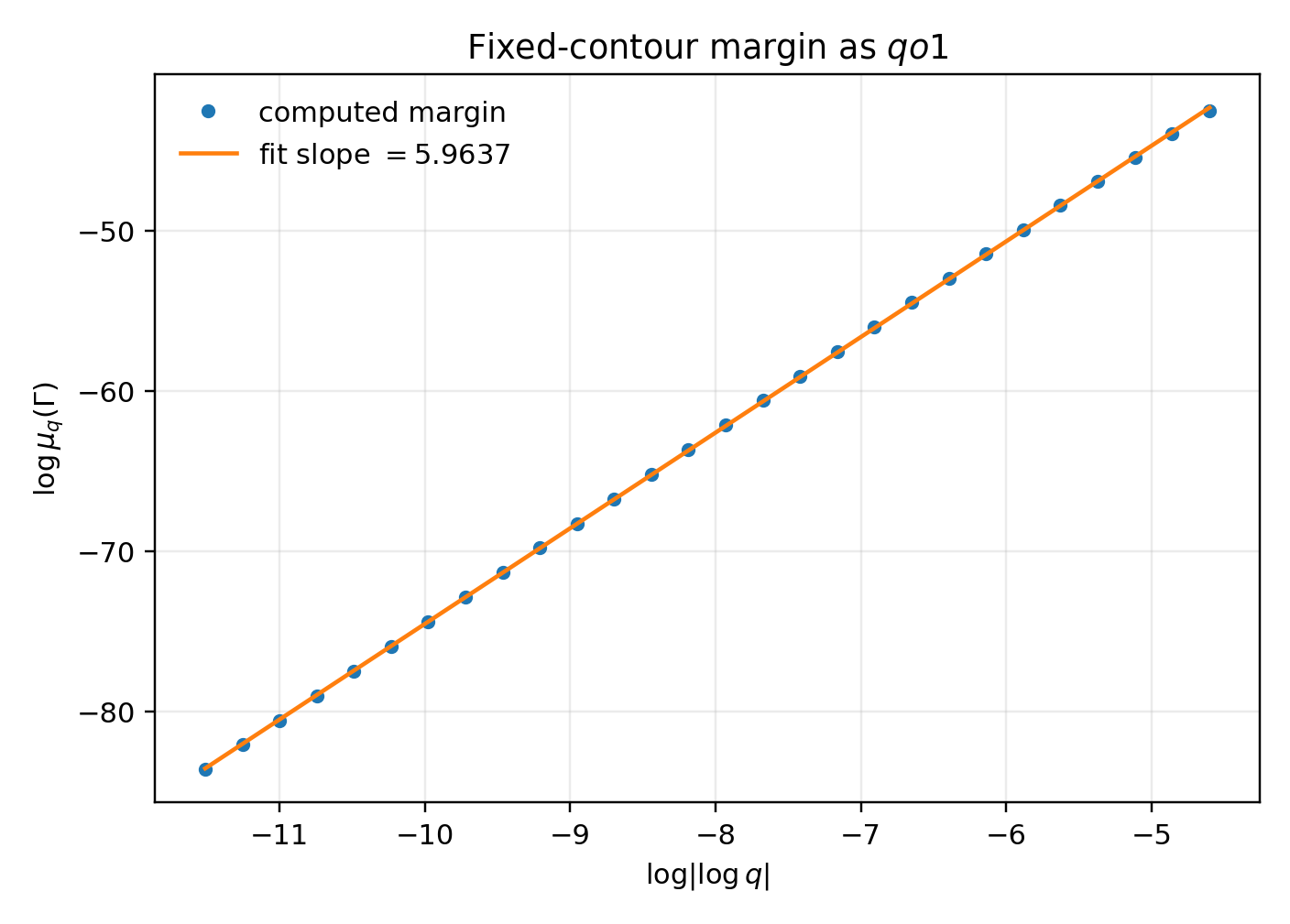}
\caption{Fixed-contour Rouch\'{e} margin as $q\to1$.  The fitted slope $5.9637$ agrees with the predicted exponent $r(r-1)=6$.}
\label{fig:exp5-margin-scaling}
\end{figure}

\section{Conclusion}
\label{sec:conclusion}

Dilation covariance provides an operator-level origin for the Hankel structure generated by geometric Mellin sampling.  The projective multiplication law fixes the scalar normalization that restores additive Hankel indexing, while specialization to a sparse Mellin spectrum recovers the classical finite Prony factorization.  At several scales, the same framework separates two effects of the sampling ratio: incommensurate logarithmic steps remove complex aliasing after cross-scale matching, whereas well-separated scales improve the numerical resolution of clustered spectra.

The determinant of the minimal Hankel pencil gives a complementary regional description of the spectrum.  Its winding number counts nodes without first resolving them individually, and Rouch\'{e}-type bounds give noise-dependent conditions under which this count persists.  The resulting point and contour channels therefore carry different information, as seen numerically when regional counts remain stable beyond the regime of accurate point recovery.  The present perturbation estimates are sufficient rather than sharp; a finer multiscale analysis of clustered Vandermonde systems would be needed for quantitative resolution limits.

\appendix
\section{A \texorpdfstring{$q$-Fock}{q-Fock} realization}
\label{sec:qfock}

The commutation relation used above also appears in the analytic realization of the Arik--Coon oscillator.  Here $q$ is the deformation parameter, not a Mellin sampling ratio.  The purpose of this appendix is only to record the resulting finite spectral problem and the two-probe identification of $q$.

On $\mathbb C[z]$, let
\begin{equation}\label{eq:bargmann-ops}
(M_zF)(z)=zF(z),
\qquad
(T_qF)(z)=F(qz),
\qquad 0<q<1.
\end{equation}
Then $T_qM_z=qM_zT_q$.  The Jackson derivative
\begin{equation}\label{eq:Dq}
(\mathfrak D_qF)(z)=\frac{F(z)-F(qz)}{(1-q)z}
\end{equation}
has the removable value $F'(0)$ at the origin and satisfies
\begin{equation}\label{eq:qosc}
\mathfrak D_qM_z-qM_z\mathfrak D_q=\Id
\end{equation}
\cite{Jackson1909,KacCheung2002}.

Write
$$
[n]_q=\frac{1-q^n}{1-q},
\qquad
[n]_q!=\prod_{k=1}^{n}[k]_q,
$$
and equip the span of
$$
e_n(z)=\frac{z^n}{\sqrt{[n]_q!}},
\qquad n\ge0,
$$
with the inner product for which $\{e_n\}$ is orthonormal.  Its completion is the Arik--Coon analytic Fock space \cite{ArikCoon1976}.  Identifying $e_n$ with $|n\rangle$, the annihilation operator $a_q=\mathfrak D_q$ satisfies
$$
a_q|0\rangle=0,
\qquad
a_q|n\rangle=\sqrt{[n]_q}\,|n-1\rangle .
$$
For $|\lambda|^2<1/(1-q)$, the normalized coherent state is
\begin{equation}\label{eq:coherent-state}
|\lambda;q\rangle
=\cN_q(|\lambda|^2)^{-1/2}
\sum_{n=0}^{\infty}
\frac{\lambda^n}{\sqrt{[n]_q!}}|n\rangle,
\qquad
\cN_q(t)=\sum_{n=0}^{\infty}\frac{t^n}{[n]_q!},
\end{equation}
and $a_q|\lambda;q\rangle=\lambda|\lambda;q\rangle$.

Consider a finite superposition
\begin{equation}\label{eq:qfock-superposition}
|\psi\rangle=\sum_{\ell=1}^{r}c_\ell|\lambda_\ell;q\rangle
\end{equation}
with pairwise distinct nodes $\lambda_\ell$.  For a probe $|\phi\rangle$,
\begin{equation}\label{eq:probe-moments}
\mu_n^{(\phi)}
:=\langle\phi|a_q^n|\psi\rangle
=\sum_{\ell=1}^{r}
c_\ell\langle\phi|\lambda_\ell;q\rangle\lambda_\ell^n.
\end{equation}
Thus each probe produces the same finite-exponential structure as the Mellin model, with probe-dependent effective amplitudes.  In particular, the associated Hankel matrix has rank equal to the number of visible distinct nodes once its size is large enough.

A single probe does not in general separate $q$ from the coefficients $c_\ell$.  Number-state probes have the explicit overlaps
\begin{equation}\label{eq:number-overlap}
\langle m|\lambda;q\rangle
=\cN_q(|\lambda|^2)^{-1/2}
\frac{\lambda^m}{\sqrt{[m]_q!}}.
\end{equation}
Set $w_\ell^{(m)}=c_\ell\langle m|\lambda_\ell;q\rangle$.  The node sets are common to the different probes, so the recovered effective amplitudes can be matched by node.

\begin{proposition}[Two-probe identification of $q$]
\label{prop:q-identification}
Assume $0<q<1$ and $|\lambda_\ell|^2<1/(1-q)$.  Let a component with $\lambda_\ell\ne0$ and $c_\ell\ne0$ be visible to the probes $|0\rangle$ and $|2\rangle$, and suppose that $\lambda_\ell$, $w_\ell^{(0)}$, and $w_\ell^{(2)}$ have been recovered from the two Hankel data sets.  Then
\begin{equation}\label{eq:weight-ratio-q}
\frac{w_\ell^{(2)}}{w_\ell^{(0)}}
=\frac{\lambda_\ell^2}{\sqrt{1+q}},
\end{equation}
and hence
\begin{equation}\label{eq:q-recovery}
q=
\frac{|\lambda_\ell|^4}
{\left|w_\ell^{(2)}/w_\ell^{(0)}\right|^2}-1.
\end{equation}
Every visible nonzero component gives the same value of $q$ in the noiseless setting.
\end{proposition}

\begin{proof}
Equation \eqref{eq:number-overlap} gives
$$
w_\ell^{(0)}
=c_\ell\cN_q(|\lambda_\ell|^2)^{-1/2},
\qquad
w_\ell^{(2)}
=c_\ell\cN_q(|\lambda_\ell|^2)^{-1/2}
\frac{\lambda_\ell^2}{\sqrt{[2]_q!}}.
$$
Since $[2]_q!=1+q$, taking the ratio gives \eqref{eq:weight-ratio-q}; taking absolute values yields \eqref{eq:q-recovery}.
\end{proof}

For a numerical check, take $q=0.62$, $\lambda_1=0.35$, $\lambda_2=0.70$, and $c=(1.10,-0.80)$.  The two probe Hankel pencils recover the common nodes, and \eqref{eq:q-recovery} gives
$$
\widehat q_1=\widehat q_2=0.620000000000000,
\qquad
\max_\ell|\widehat q_\ell-q|=1.67\times10^{-15}.
$$
The ratio in \eqref{eq:weight-ratio-q} removes both the superposition coefficient and the coherent-state normalization.

The Mellin and $q$-Fock realizations therefore share the same commutation law but encode inverse information differently.  In the Mellin model the sampling ratio changes the nodes and their aliasing; in the Fock realization the deformation parameter enters the probe-dependent amplitudes.

\end{document}